\documentclass[11pt, leqno]{article}
\date{}
\usepackage{amssymb, amsbsy, amsmath, amsfonts, amssymb, amscd, mathrsfs, amsthm, dsfont}
\usepackage[english]{babel}
\usepackage[T1]{fontenc}
\usepackage{indentfirst}
\usepackage{color}
\usepackage{graphicx} 

\makeatletter
\@addtoreset{equation}{section}
\makeatother

\newtheorem{statement}{}[section]
\newtheorem{theorem}[statement]{Theorem}
\newtheorem{lemma}[statement]{Lemma}

\newtheorem{proposition}[statement]{Proposition}

\newtheorem{corollary}[statement]{Corollary}

\newcommand\C{\mathbb C}
\newcommand\N{\mathbb N}

\newcommand\T{\mathbb T}
\newcommand\D{\mathbb D}

\newcommand\e{{\rm e}}

\renewcommand \Re{{\mathfrak R}{\rm e}\,}
\renewcommand \Im{{\mathfrak I}{\rm m}\,}
\let\phi=\varphi
\newcommand\HH{\mathbb H}

\title{\bf An extremal composition operator on the Hardy space of the bidisk with small approximation numbers}
\author{\it Daniel Li,  Herv\'e Queff\'elec, Luis Rodr{\'\i}guez-Piazza }

\date{\footnotesize \today}

\begin{document}

\maketitle

\noindent {\bf Abstract.} \emph{We construct an analytic self-map $\Phi$ of the bidisk $\D^2$ whose image touches the distinguished boundary, but whose 
approximation numbers of the associated composition operator on $H^2 (\D^2)$ are small in the sense that 
$\limsup_{n \to \infty} [a_{n^2} (C_\Phi)]^{1 / n} < 1$.}
\medskip

\medskip


%%%%%%%%%%%%%%%%%%%%%%%%%%%%%%%%%%%%%%%%%%%%%%%%%%%%%%%%%%%%%%%%%%%%%%%%%%%%
\section{Introduction}

For composition operators $C_\Phi \colon H^2 (\D) \to H^2 (\D)$ on the Hardy space of the unit disk, the decay of their approximation numbers 
$a_n (C_\Phi)$ cannot be arbitrarily fast, and actually cannot supersede a geometric speed (\cite{Parfenov}; see also \cite[Theorem~3.1]{LIQUEROD}): there 
exists a positive constant $c$ such that:
\begin{displaymath}
\qquad\qquad\quad a_n (C_\Phi) \gtrsim \e^{ - c n} \, , \qquad n = 1, 2, \ldots 
\end{displaymath}
It is easy to see that this speed occurs when $\| \Phi \|_\infty < 1$, and we proved in \cite[Theorem~3.4]{LIQUEROD} that a geometrical speed only takes 
place in this case; in other words:
\begin{equation} \label{hada} 
\qquad \Vert \Phi \Vert_\infty = 1 \quad \Longleftrightarrow \quad \lim_{n \to \infty} [a_n (C_\Phi)]^{1 / n} = 1 \, .
\end{equation}

This leads to the introduction, for an operator $T$ between Banach spaces, of the parameters:
\begin{equation}
\beta^{-} (T) = \liminf_{n \to \infty}\, [a_n (T)]^{1 / n} \quad \text{and} \quad \beta^{+} (T) = \limsup_{n \to \infty}\, [a_n (T)]^{1 / n} \, ,
\end{equation}
where $a_n (T)$ is the $n$-th approximation number of $T$. When $[a_n (T)]^{1 / n}$ actually has a limit, i.e. when $\beta^{-} (T) = \beta^{+} (T)$, we write 
it $\beta (T)$. 

What is proved in \cite[Theorem~3.4]{LIQUEROD} is that $\beta (C_\Phi) = 1$ if and only if $\| \Phi \|_\infty = 1$. Later, in \cite{SRF}, we gave, when 
$\| \Phi \|_\infty < 1$,  a formula for this parameter in terms of the Green capacity of $\Phi (\D)$, which allowed us to recover \eqref{hada}.
\smallskip

More generally, for $N \geq 1$, we introduce:
\begin{equation} \label{coin-coin} 
\beta_{N}^{-} (T) = \liminf_{n \to \infty} \, [a_{n^N}(T)]^{1/n} \quad \text{and} \quad 
\beta_N^{+} (T) = \limsup_{n \to \infty}\, [a_{n^N} (T)]^{1 / n} \, ,
\end{equation}
and:
\begin{equation} \label{coin} 
\beta_{N}(T) = \lim_{n\to \infty} [a_{n^N}(T)]^{1/n} 
\end{equation} 
when the limit exists. It is clear that $0\leq \beta_{N}^{\pm}(T)\leq 1$, and it is interesting to know when the extreme cases $\beta_{N}^{\pm}(T) = 0$ or 
$\beta_{N}^{\pm} (T) = 1$  occur. For example:
\begin{align*}
\beta_{N}^{-} (T) > 0 \quad & \Longleftrightarrow \quad a_{n^N} (T) \gtrsim \e^{- \tau n} \, , \quad \text{ with } \tau > 0 \smallskip \\
\beta_{N}^{-} (T) = 1 \quad & \Longleftrightarrow \quad a_{n^N} (T) \gtrsim \e^{- n \varepsilon_n} \, , \quad \text{with }  \varepsilon_n \to 0 \, .
\end{align*}

It is coined in \cite{BLQR} (see also \cite{DHL} and \cite{DHL-pluricap}) that $\beta_N^{\pm} (C_\Phi)$ are the suitable parameters for the composition 
operators on $H^2 (\D^N)$, and it is proved, for any $N \geq 1$, that $\beta_{N}^{-} (C_\Phi) > 0$, as soon as $\Phi$ is non degenerate (i.e. the Jacobian 
$J_\Phi$ is not identically $0$) and the operator $C_\Phi$ is bounded on $H^{2} (\D^N)$. As for an expression of $\beta_{N}^{\pm} (C_\Phi)$ in terms of 
``capacity'', only partial results are known so far (\cite{DHL} and \cite{DHL-pluricap}) and the application to a result like \eqref{hada} fails in general. We gave 
an example of such a phenomenon in \cite[Theorem~5.12]{DHL}. In the present paper we give a shaper result. 

%%%%%%%%%%%%%%%%%%%%%%%%%%%%%%%%%%%%%%%%%%%%%%%%%%%%%%%%%%%%%%%%%%%%%%%%%%%%%%
\section{Background and notation} 

Let $\D$ be the open unit disk, $H^{2} (\D^N)$ the Hardy space of the polydisk $\D^N$, and $\Phi \colon \D^N \to \D^N$ an analytic map. When $N = 1$, 
it is well-known (see \cite{COMA} or \cite{SHA}) that $\Phi$ induces a composition operator $C_\Phi \colon H^{2} (\D) \to H^{2} (\D)$ by the formula:
\begin{displaymath}
C_{\Phi} (f) = f \circ\Phi \, ,
\end{displaymath}
and the connection between the ``symbol'' $\Phi$ and the properties of the operator $C_{\Phi}$, in particular its compactness, can be further studied 
(see \cite{COMA} or \cite{SHA}). When $N > 1$, $C_\Phi$ is not bounded in general (see \cite{COMA}).

Let $\T$ be the unit circle, and $m$ the normalized Haar measure on $\T^N$. A positive Borel measure $\mu$ on $\D^N$ is called a Carleson measure (for the 
space $H^{2} (\D^N)$) if the canonical injection $J \colon H^{2} (\D^N) \to L^{2} (\mu)$ is bounded. When $\Phi \colon \D^N \to \D^N$ is analytic and 
induces a bounded composition operator on $H^{2}(\D^N)$, the pullback measure $m_\Phi = \Phi^{\ast} (m)$, defined, for any test function $u$, by:
\begin{displaymath}
\int_{\D^N} u (w) \, dm_{\Phi} (w) = \int_{\T^N} u [\Phi^{\ast} (\xi)] \, dm (\xi) \, ,
\end{displaymath}
is a Carleson measure. Here $\Phi^\ast$ is the radial limit function, defined for $m$-almost every $\xi \in \T^N$, by 
$\Phi^{\ast} (\xi) = \lim_{r \to 1^{-}} \Phi (r \xi)$. 
\smallskip

For $\xi \in \T = \partial{\D}$ and $h > 0$, the Carleson window $S (\xi, h)$ is defined as:
\begin{equation}
S (\xi, h) = \{z \in \D \, ; \ |z - \xi| \leq h \} \, . 
\end{equation}

If $f \in \text{Hol}\, (\mathbb{D}^{2})$, $D_{\! j}^{\, k} f$ denotes the $k$-th derivative of $f$ with respect to the $j$-th variable ($j =1, 2$).
\smallskip

We denote by $A (\D)$ the disk algebra, i.e. the space of functions holomorphic in $\D$ and continuous on $\overline{\D}$. We similarly define the bidisk 
algebra  $A (\D^2)$.
\medskip

Let $H_1$ and $H_2$ be Hilbert spaces, and $T \colon H_1 \to H_2$ an operator. The $n$-th approximation number $a_{n}(T)$ of  $T$, $n = 1,2, \ldots$, is 
defined (see \cite{CAST}) as the distance (for the operator-norm) of $T$ to operators of rank $ < n$:
\begin{equation}
a_n (T) = \inf_{\text{rank}\, R < n} \| T - R \| \, .
\end{equation}

The approximation numbers have the ideal property:
\begin{displaymath}
a_{n} (A T B) \leq \Vert A \Vert \, a_{n} (T) \, \Vert B \Vert \, . 
\end{displaymath}

The $n$-th Gelfand number $c_{n} (T)$ of $T$ is defined by:
\begin{equation}
c_{n} (T) = \inf_{\text{codim}\, E < n} \Vert T_{\mid E} \Vert \, .
\end{equation}
As an easy consequence of the Schmidt decomposition, we have for any compact operator between Hilbert spaces:
\begin{equation}
c_{n} (T) = a_{n}(T) \, .
\end{equation}

If $T, T_1, T_2 \colon H \to H'$ are operators between Hilbert spaces $H$ and $H'$, we write $T = T_1 \oplus T_2$ if $T = T_1 + T_2$ and:
\begin{displaymath}
\qquad \quad \Vert T x \Vert^2 = \Vert T_{1} x \Vert^2 + \Vert T_{2} x \Vert^2 \, , \quad \text{for all } x \in H\, .
\end{displaymath}

The subaddivity of approximation numbers is then expressed by:
\begin{equation} \label{sa} 
a_{j + k} (T_1\oplus T_2) \leq a_{j} (T_1) + a_{k} (T_2) \, .
\end{equation}

We denote by $\N = \{0, 1, 2, \ldots\}$ the set of non-negative integers, and by $[x]$ the integral part of the real number $x$.
\smallskip

We write $X \lesssim Y$ to indicate that $X \leq c\, Y$ for some constant $c > 0$, and $X \approx Y$ to indicate that $X \lesssim Y$ and $Y \lesssim X$.

%%%%%%%%%%%%%%%%%%%%%%%%%%%%%%%%%%%%%%%%%%%%%%%%%%%%%%%%%%%%%%%%%%%%%%%%%%%%%
\section{Purpose of the paper}

Let us recall that the Hardy space of the polydisk is the space:
\begin{displaymath}
H^{2} (\D^N) = \Big\{f \colon \D^N \to \C \, ; \  f (z) = \sum_{\alpha\in \N^{N}} a_\alpha z^\alpha \text{ and } 
\Vert f \Vert_{2}^{2}:=\sum |a_\alpha|^2 < \infty \Big\} \, .
\end{displaymath}

If $\Phi \colon \D^N \to \D^N$ is an analytic map, the associated composition operator $C_\Phi$ (which is not always bounded on $H^{2}(\D^N)$) is 
defined by:
\begin{displaymath}
C_{\Phi} (f) = f \circ \Phi \, .
\end{displaymath}

We will mainly here be interested in the case $N = 2$. 
\smallskip

The reproducing kernel $K_a$ of $H^{2} (\D^2)$ is, with $a = (a_1, a_2)$ and $z = (z_1, z_2)$:
\begin{equation}
K_{a} (z) = \frac{1}{(1 - \overline{a_1} z_1) (1 - \overline{a_2} z_2)} \, \cdot  
\end{equation}
As a consequence:
\begin{equation} \label{asa} 
| f (a) | = | \langle f, K_a \rangle | \leq \frac{\Vert f \Vert_2}{\sqrt{(1 - |a_1|^2) (1 - |a_2|^2)}} \, \cdot
\end{equation}
In particular, the functions in the unit ball of $H^2 (\D^2)$ are uniformly bounded on compact subsets of $\D^2$. 
\smallskip

In \cite[Theorem~5.12]{DHL}, we gave an example of a holomorphic self-map $\Phi \colon \D^2 \to \D^2$, continuous on the closure $\overline{\D^2}$, 
such that $\Vert \Phi \Vert_\infty = 1$, that is:
\begin{equation} \label{one} 
\Phi (\T^2) \cap \partial{\D^2} \neq \emptyset \, ,
\end{equation}
and yet:
\begin{equation}
\beta_{2}^{+} (C_\Phi) < 1 \, ,
\end{equation}
in contrast with the one-dimensional case (\cite[Theorem~3.4]{LIQUEROD}). \goodbreak

Understanding where the difference really lies when we pass to the multidimensional case is a big challenge: it does not seem to be a matter of regularity of the 
boundary, and a similar example probably holds for the Hardy space of the ball. It might be a matter of  boundary: the Shilov boundary of the ball is its usual 
boundary, but that of the polydisk is its distinguished boundary:
\begin{displaymath}
\partial_{e} {\D^N} = \{z = (z_j) \, ; \ |z_j| = 1 \text{ for all }  j = 1, \ldots, N \} = \T^N 
\end{displaymath}
(indeed, the distinguished maximum principle tells that, for $f$ analytic in $\D^N$ and continuous on $\overline{\D^N}$, it holds 
$\max_{z \in \D^N} | f (z) | = \max_{z \in \partial_{e} {\D^N}} | f (z) |$). The aim of this paper is to show that this is not the case and, improving on 
(\cite[Theorem~5.12]{DHL}) and \eqref{one}, to build an analytic self-map $\Phi \colon \D^2 \to \D^2$, continuous on $\overline{\D^2}$, non-degenerate 
and such that:
\begin{equation} \label{two} 
\Phi (\T^2) \cap \partial_{e}{\D^2} \neq \emptyset \, , \quad \text{but} \quad  \beta_{2}^{+} (C_\Phi) < 1 \, .
\end{equation}

The paper is organized as follows. In Section~\ref{section cusp}, we recall with some detail the definition and main properties of a so-called \emph{cusp map} 
$\chi \in A (\D)$, to be of essential use in our counterexample. In Section~\ref{lemmas}, we prove several lemmas which constitute the core or the proof. In 
Section~\ref{section main}, we state and prove our main theorem.

%%%%%%%%%%%%%%%%%%%%%%%%%%%%%%%%%%%%%%%%%%%%%%%%%%%%%%%%%%%%%%%%%%%%%%%
\section{The cusp map} \label{section cusp}

The \emph{cusp map} $\chi \colon \D \to \D$ is analytic in $\D$ and extends continuously on $\overline{\D}$. The boundary of its image is formed by three 
circular arcs of respective centers $\frac{1}{2}$, $1 + \frac{i}{2}$, $1 - \frac{i}{2}$, and of radius $\frac{1}{2}$ (see Figure~\ref{cusp map}). However, 
the parametrization $t \mapsto  \chi (\e^{i t})$ involves logarithms. 

\begin{figure}[ht]
\centering
\includegraphics[width=6cm]{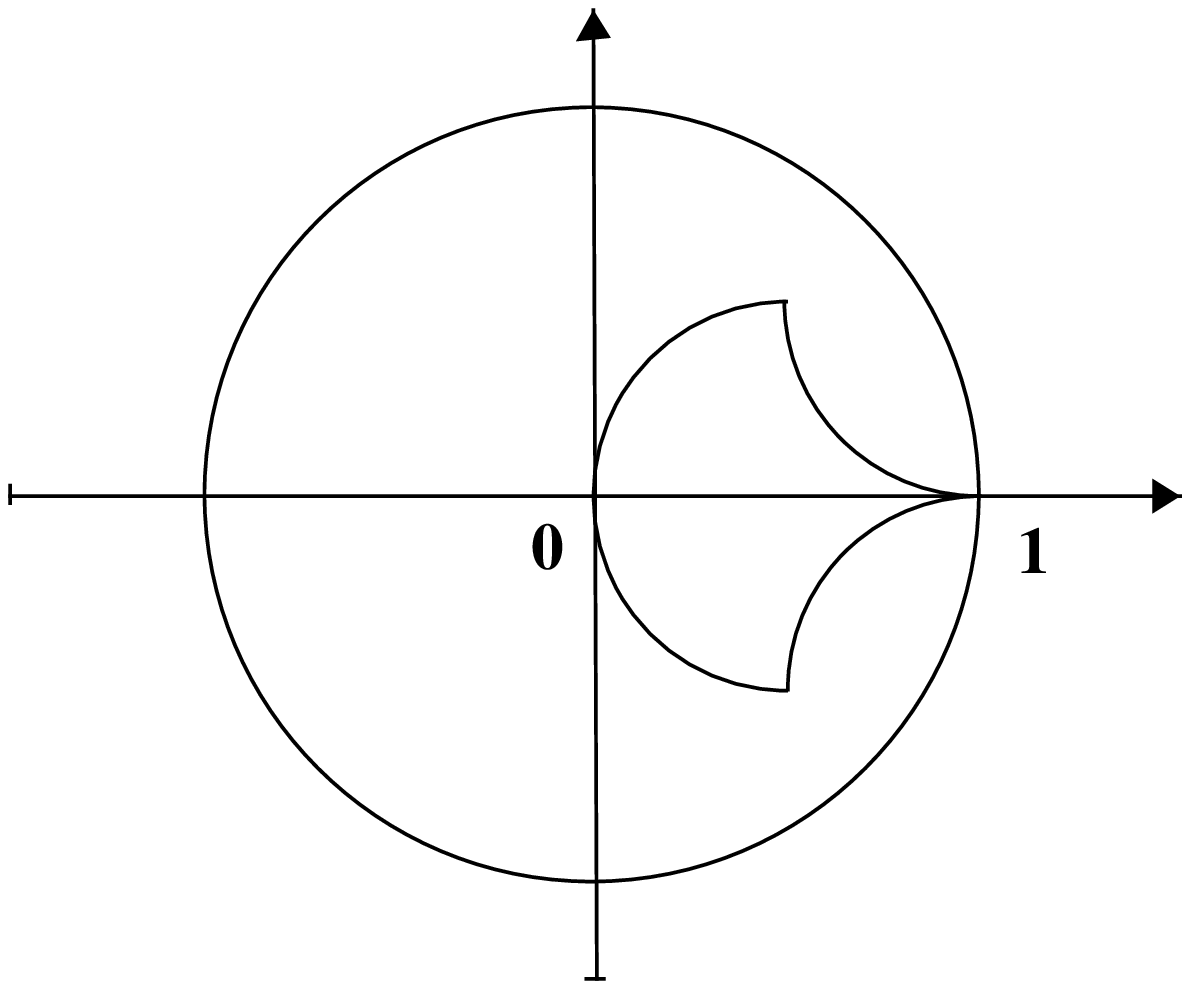}
\caption{\it Cusp map domain} \label{cusp map}
\end{figure} 

It was often used by the authors (\cite{LIQURO}, \cite{PDHL}) as an extremal example. 

We first recall the definition of $\chi$. 

Let $\D^+ = \{z \in \D \, ; \ \Re z > 0 \}$ be the right half-disk. Let now $\HH$ be the upper half-plane, and $T \colon \D \to \HH$ defined by:
\begin{displaymath}
T (u) = i \, \frac{1 + u}{1 - u} \, , \quad \text{with} \quad T^{- 1} (s) = \frac{s - i}{s + i} \, \cdot
\end{displaymath}
Taking the square root of $T$, we map $\D$ onto the first quadrant defined by $Q_1 = \{z \in \C \, ; \ \Re z > 0\}$; we go back to the half-disk 
$\{z \in \D \, ; \ \Im z < 0\}$ by $T^{- 1}$. Finally, make a rotation by $i$ to go onto $\D^+$. We get:
\begin{displaymath}
\chi_{0} (z) = \frac{\displaystyle \phantom{-} \Big(\frac{z - i}{i z - 1} \Big)^{1/2} - i}{\displaystyle -i \Big( \frac{z - i}{i z - 1} \Big)^{1/2} + 1} \, \cdot 
\end{displaymath}
One has $\chi_{0} (1) = 0$, $\chi_{0} (- 1) = 1$, $\chi_{0} (i) = - i$, and $\chi_{0} (- i) = i$. The half-circle $\{z \in \T \, ; \ \Re z \geq 0\}$ is mapped by 
$\chi_0$ onto the segment $[-i, i]$ and the segment $[- 1, 1]$ onto the segment $[0, 1]$. 

Set now, successively:
\begin{equation} \label{suc} 
\chi_{1} (z) = \log \chi_{0}(z) \, , \quad  \chi_{2} (z) = - \frac{2}{\pi} \, \chi_{1} (z) + 1 \, , \quad \chi_{3} (z) = \frac{1}{\chi_{2} (z)} 
\end{equation}
and finally:
\begin{equation} \label{fin} 
\chi (z) = 1 - \chi_{3} (z) \, .
\end{equation}

We now summarize the properties of the cusp map $\chi$ in the following proposition.

\begin{proposition} \label{cusp} 
The cusp map satisfies:
\begin{enumerate}
\setlength\itemsep{2 pt}

\item [$1)$] $\displaystyle 1  - |\chi(z)| \lesssim \frac{1}{\log (2/|1 - z|)}$; 

\item [$2)$] $|1 - \chi (z)| \leq K ( 1 - |\chi (z) |)$ for all $z \in \D$, where $K$ is a positive constant;

\item [$3)$] $\chi(\D)$ is the intersection of the open disk $D \big( \frac{1}{2} \raise 1,5 pt \hbox{,} \, \frac{1}{2} \big)$ with the exterior of the two open 
disks $D \big(1 + \frac{i}{2} \raise 1,5 pt \hbox{,} \, \frac{1}{2} \big)$ and 
$D \big( 1 - \frac{i}{2} \raise 1,5 pt \hbox{,} \, \frac{1}{2} \big)$;

\item [$4)$] $\chi (1) = 1$, $\chi (\overline{z}) = \overline{\chi (z)}$ and $| \chi (z) - 1| \leq 1$ for all $z \in \D$;

\item [$5)$] for $0 < |t| \leq \pi / 4$, we have $1 -  \Re \chi (\e^{i t}) \approx 1 / (\log 1/|t|)$;

\item [$6)$] $\chi (\overline{\D}) \subseteq \{z = x + i y \, ;\ 0 \leq x \leq 1 \text{ and } |y| \leq 2 (1 - x)^2 \}$. 
\end{enumerate}
\end{proposition}
\begin{proof}
Items $1)$ to $5)$ are proved in \cite[Lemma~4.2]{LIQURO}. To prove $6)$, write $\chi (z) = (1- h) + i y$. Since $\chi (\overline{z}) = \overline{\chi (z)}$, 
we can assume $y \geq 0$. Since $\chi (\D) \cap D \big( 1 + \frac{i}{2} \raise 1,5 pt \hbox{,} \, \frac{1}{2} \big) = \emptyset$, we have  
$\big| \chi (z) - \big( 1 + \frac{i}{2} \big) \big| \geq \frac{1}{2}$; hence:
\begin{displaymath}
h^2 + \Big(y - \frac{1}{2} \Big)^2 = \bigg| \chi (z) - \Big( 1 + \frac{i}{2} \Big) \bigg|^2 \geq \frac{1}{4} \, ,
\end{displaymath}
so that $y \leq y^2 + h^2$. But $y \leq 1/2$, since $\chi (z) \in D \big( \frac{1}{2} \raise 1,5 pt \hbox{,} \, \frac{1}{2} \big)$; therefore $y^2 \leq y / 2$, 
so we get $y \leq 2 h^2$.
\end{proof}
%

%%%%%%%%%%%%%%%%%%%%%%%%%%%%%%%%%%%%%%%%%%%%%%%%%%%%%%%%%%%%%%%%%%%%%%%%
\section{Preliminary lemmas} \label{lemmas} 

In this section, we collect some lemmas, which will reveal essential  in the proof of our counterexample. 
\smallskip

We consider the map $\varphi = \varphi_\theta$, $0 < \theta < 1$, defined, for $z \in \overline \D \setminus \{1 \}$, by:
\begin{equation} \label{def} 
\varphi \, (z) = \exp \big( - (1 - z)^{- \theta} \big) \, . 
\end{equation} 

We observe, since $\Re (1 - z) \geq 0$ for $z \in \overline \D$, that:
\begin{equation} \label{ob} 
|\varphi \, (z)| \leq \exp \big( - \delta \, |1 - z|^{- \theta} \big) \, ,
\end{equation}
where $\delta = \cos \pi \theta / 2 > 0$. Moreover, \eqref{ob} shows that $\varphi \in A (\D)$, since:
\begin{displaymath} 
\lim_{z \to 1, z \in \overline{\D}} \varphi \, (z) = 0 =: \varphi \, (1) \, .
\end{displaymath} 

Our first lemma will allow us to define our symbol $\Phi$.

\begin{lemma} \label{eins} 
One can adjust $0 < c < 1$ so as to get:
\begin{equation} \label{un} 
\qquad \qquad \quad |\chi (z)|+ 2 \, c \, | \varphi \circ \chi (z) | < 1 \qquad \text{for all } z \in \D \, . 
\end{equation} 
Hence, if we set, for any $g \in A (\D)$ with $\| g \|_\infty \leq 1$:
\begin{equation} \label{def-Phi}
\Phi (z_1, z_2) = \big( \chi (z_1), \chi (z_1) + c \, (\phi \circ \chi) (z_1) \, g (z_2) \big) \, , 
\end{equation} 
we have $\Phi (\D^2) \subseteq \D^2$.
\end{lemma}
\noindent {\bf Remark.} The factor $2$ in \eqref{un} is needed in order to get the following inequalities, to be used later, for $z \in \D$ and 
$w = \chi (z) + c \, (\varphi \circ\chi) (z) \,u$, with $|u| \leq 1$:
\begin{equation} \label{deux} 
|w|\leq \frac{1 + |\chi (z)|}{2} \, \raise 1,5 pt \hbox{,}
\end{equation}
or, equivalently:
\begin{equation} \label{deux-et-demi}
1 - |w| \geq \frac{1 - |\chi (z)|}{2} \, \cdot
\end{equation} 
Indeed:
\begin{displaymath} 
|w| \leq  |\chi (z)| + c \, |\varphi \circ \chi (z)|\leq  |\chi (z)| + \frac{1 - |\chi (z)|}{2} = \frac{1 + |\chi (z)|}{2} \, \cdot
\end{displaymath} 
\begin{proof} [Proof of Lemma~\ref{eins}] 
Set $X = |1 - \chi (z)|$, so that, with $K$ the constant of Proposition~\ref{cusp},\,$2)$:
\begin{equation} \label{mar} 
| \chi (z) | \leq 1 - \frac{|1 - \chi (z)|}{K} = 1 - \frac{X}{K} \, \cdot
\end{equation}
For $z \in \D$ and $X$ close enough to zero, say $X < \eta$,  we have $2 \exp (- \delta X^{-\theta}) < \frac{X}{K}\,$. If we adjust $0 < c < 1$ so as to 
have $c < \frac{\eta}{2K}$, it follows from \eqref{ob} and \eqref{mar} that, for $X< \eta$:
\begin{displaymath} 
|\chi (z)|+ 2 \, c \, |\varphi \circ \chi (z)|\leq 1 - \frac{X}{K} + 2 \exp (- \delta X^{- \theta}) < 1 \, .
\end{displaymath} 
However, for $X \geq \eta$, \eqref{mar} says that $|\chi (z) | \leq 1 - \frac{\eta}{K}$, so:
\begin{displaymath} 
|\chi (z)|+ 2 \, c\, |\varphi \circ \chi (z)|\leq 1 - \frac{\eta}{K}+ 2 \, c < 1 \, ,
\end{displaymath} 
as well and this ends the proof of  Lemma \ref{eins}. 
\end{proof}

Our second lemma estimates some integrals and ensures that $\Phi$ induces a compact composition operator on $H^2 (\D^2)$.

\begin{lemma}\label{zwei} 
For $0 < h \leq 1$, the following estimate holds:
\begin{equation} \label{I-zero}
I_{0} (h) := \int_{|\chi (\e^{i t}) - 1| \leq h} \frac{1}{(1 - |\chi (\e^{i t})|)^2} \, dt \lesssim \e^{- \tau / h} \, ,
\end{equation} 
\end{lemma}
\begin{proof}
By Proposition~\ref{cusp},\,$5)$, there exist two constants $c_1$, $c_2$ such that:
\begin{displaymath} 
\qquad \qquad \frac{c_1}{\log 1/|t|} \leq |\chi (\e^{i t}) - 1| \leq \frac{c_2}{\log 1/|t|} \, , \qquad |t| \leq \pi \, ;
\end{displaymath} 
hence:
\begin{displaymath} 
I_{0} (h)  \lesssim \int_{|t| \leq \e^{- c_1/h}} [\log (1/|t|)]^2 \, dt  
= 2 \int_{c_1 / h}^\infty x^2 \, \e^{- x} \, dx \lesssim h^{- 2} \e^{- c_1/h} \, . \qedhere
\end{displaymath} 
\end{proof}
\begin{corollary} \label{Phi-HS} 
For $g \in A (\D)$ with $0 < \Vert g \Vert_\infty \leq 1$, set:
\begin{displaymath} 
I (h) := \int_{| \chi ( \e^{i t_1}) - 1| \leq h} 
\frac{d t_1 \, dt_2}{(1 - |\chi (\e^{i t_1})|)(1 - |\chi (\e^{i t_1}) + c \, (\varphi \circ \chi) (\e^{i t_1}) \, g (\e^{i t_2})|)} \, \cdot
\end{displaymath} 
Then:
\begin{displaymath} 
I (h)  \lesssim \e^{- \tau/ h} \, .
\end{displaymath} 
Consequently, the composition operator $C_\Phi$ defined in \eqref{def-Phi} is bounded from $H^{2}(\D^2)$ to $H^{2}(\D^2)$ and is compact.
\end{corollary}
\begin{proof}
Using \eqref{deux-et-demi}, we have, thanks to \eqref{I-zero}:
\begin{displaymath} 
I (h) \leq \int_{|\chi (\e^{i t_1}) - 1|\leq h} \frac{2}{(1 - |\chi (\e^{i t_1})|)^2} \, d t_1 \, dt_2 \lesssim \e^{- \tau / h} \, .
\end{displaymath} 
In particular, $I (1) < \infty$, showing that $C_\Phi$ is Hilbert-Schmidt and hence bounded.
\end{proof}
\smallskip

For the rest of the paper, we fix a number $\sigma$ in $(0, 1)$, that for convenience we take as:
\begin{equation} \label{sigma}
\sigma = \frac{7}{8} \, \raise 1,5 pt \hbox{,}
\end{equation} 
a positive integer $j_0$ such that:
\begin{equation} \label{j-zero}
2 \, \sigma^{j_0}\leq 1/8
\end{equation} 
(i.e. $j_0 \geq 21$), and we set:
\begin{equation} 
a_j = 1 - \sigma^{j}
\end{equation} 
and:
\begin{equation} 
\qquad \rho_j = \frac{\sigma^{j}}{4} = \frac{1}{4} (1 - a_j) \, .
\end{equation} 
We also define, for $n \geq 1$ and $\theta$ being the parameter used in \eqref{def}:
\begin{equation} 
N_n = \bigg[ \frac{\log 2 n}{\theta \log 1/\sigma } \bigg] + 1 > \frac{\log 2 n}{ \log 1/\sigma} \, \cdot
\end{equation} 

The next lemma gives a cutting off for $\chi (\D)$.

\begin{lemma} \label{drei}
For every $n \geq 1$, the image $\chi (\D)$ of the cusp map, deprived of the closed Euclidean disk $\overline{D} (0, 1 - \sigma^{j_0} /K )$ and of  
$\chi (\D) \cap S(1, 1/n)$, can be covered by the open Euclidean disks $D (a_j, \rho_j)$, with $j_{0}\leq j \leq N_n$. 
\end{lemma}
\begin{proof}
Let $z \in \D$ such that $|\chi (z)|> 1 - \sigma^{j_0} / K$ and $|\chi (z) - 1| > 1/n$. We write $\chi (z) = x + i y =: 1 - h + i y$. 

Let $j$ with $a_j \leq x < a_{j + 1}$, i.e. $\sigma^{j +1} < h \leq \sigma^j$. We have $j \geq j_{0}$, since $h < \sigma^{j_{0}}$. 

Now, since $0 \leq x - a_j < a_{j + 1} - a_j  = \sigma^{j + 1} - \sigma^j$, that $y^2 \leq 4 h^4$ (Proposition~\ref{cusp},\,$6)$), and $h \leq \sigma^j$, 
we have:
\begin{displaymath} 
|\chi (z) - a_j|^2 < (\sigma^{j} - \sigma^{j + 1})^2 + y^2 \leq (1 - \sigma)^2 \sigma^{2 j} + 4 \, \sigma ^{4 j} \, ;  
\end{displaymath} 
hence:
\begin{displaymath} 
|\chi (z) - a_j| < \sigma^j (1 - \sigma) + 2 \, \sigma^{2 j} = \sigma^j (1 - \sigma + 2 \, \sigma^j)  \, .
\end{displaymath} 
Subsequently, since $1 - \sigma = 1/8$, $j \geq j_0$,  and $2 \, \sigma^{j_0} \leq 1/8$:
\begin{displaymath} 
|\chi (z) - a_j| < \sigma^j (1 - \sigma + 2 \, \sigma^{j_0}) \leq  \frac{\sigma^j}{4\,} = \rho_j \, , 
\end{displaymath} 
showing that $\chi (z) \in D (a_j, \rho_j)$. 

Moreover, we have $j \leq N_n$. Indeed, if $j > N_n$, we would have:
\begin{displaymath} 
|\chi (z) - 1| \leq |\chi (z) - a_j|+ (1 - a_j) \leq \rho_j + \sigma^j = \frac{5}{4} \, \sigma^j \leq \frac{5}{4}\, \sigma^{N_n + 1} 
\leq 2\, \sigma^{N_n} \leq 1/n \, ,
\end{displaymath} 
contradicting the fact that $\chi (z) \notin S (1, 1 /n)$. 
\end{proof}
\smallskip

Our next two lemmas give estimates on derivatives for the functions belonging to $H^2 (\D^2)$.

\begin{lemma} \label{vier} 
Let $f \in H^{2} (\D^2)$, $k$ a non-negative integer, $b \in \D$, and let  $h_k (z) = (D_{2}^{\, k}f) (z, z)$. Then:
\begin{displaymath} 
|h_{k} (b)| \leq \frac{k! \, 2^{k + 1}}{(1 - |b|)^{k + 1}} \, \Vert f \Vert_2 \, .
\end{displaymath} 
\end{lemma}
\begin{proof}
The Cauchy inequalities give for $0 < s < 1 - |b|$ and $\alpha \in \N^2$: 
\begin{displaymath} 
|D^{\alpha} f (b, b)| \leq \frac{\alpha!}{s^{|\alpha|}} \sup_{|w_1 - b| = s, |w_2 - b| = s} |f (w_1, w_2)| \, .
\end{displaymath} 
The choice $s = \frac{1 - |b|}{2}$ gives $1-|w_j|\geq \frac{1-|b|}{2}$ for $|w_j - b| = s$, $j = 1, 2$; hence, thanks to the estimate \eqref{asa}: 
\begin{displaymath} 
|f (w_1, w_2)| \leq   \frac{\Vert f \Vert_2}{\sqrt{(1 - |w_1|)(1 - |w_2|)}} \leq \frac{2}{1 - |b|} \, \Vert f \Vert_2 \, \cdot
\end{displaymath} 
Specializing to $\alpha = (0, k)$ now gives the result.
\end{proof}

\begin{lemma} \label{funf}
With the notations of Lemma~\ref{vier}, assume that $h_{k}^{(l)} (a) = 0$ for some $a \in \D$ and for $0 \leq l < n$. Then, for $0< \rho < 1$ and 
$|b - a| \leq \frac{\rho}{2} \, (1 - |a|)$, it holds: 
\begin{displaymath} 
|h_{k} (b)| \leq \rho^n \, \frac{k! \, 4^{k + 1}}{(1 - |a|)^{k + 1}} \, \Vert f \Vert_2 \, .
\end{displaymath} 
\end{lemma}
\begin{proof} 
We may assume $\Vert f \Vert_2 \leq 1$. Consider the function defined, for $w \in \D$,  by:
\begin{displaymath} 
H_{k} (w) = h_{k} \bigg( a + w \, \frac{1 - |a|}{2} \bigg) \, .
\end{displaymath} 
It is a bounded and holomorphic function in $\D$. 

For $w \in \D$, let $\beta = a + w \, \frac{1 - |a|}{2}$, which satisfies $1 - |\beta| \geq \frac{1 - |a|}{2}\,$. Lemma~\ref{vier} gives:
\begin{displaymath} 
|H_{k} (w)| = |h_{k} (\beta)| \leq  \frac{k! \, 4^{k + 1}}{(1 - |a|)^{k + 1}}  \, \cdot
\end{displaymath} 
Now, $H_{k}^{(l)} (0) = h_{k}^{(l)} (a) = 0$ for $0 \leq l < n$; hence the Schwarz lemma says that $H_k$ satisfies 
$|H_{k} (w)| \leq |w|^n \, \Vert H_k \Vert_\infty$ for all $w \in \D$.  Take $w = \frac{2 (b - a)}{1 - |a|}\,$, which satisfies $|w| \leq \rho$, to get:
\begin{displaymath} 
|h_{k} (b)| =|H_{k}(w)| \leq |w|^n \, \Vert H_k \Vert_\infty \leq \rho^n  \, \frac{k! \, 4^{k + 1}}{(1 - |a|)^{k + 1}} \, \cdot \qedhere
\end{displaymath} 
\end{proof}
%

%%%%%%%%%%%%%%%%%%%%%%%%%%%%%%%%%%%%%%%%%%%%%%%%%%%%%%%%%%%%%%%%%%%%%%%%%
\section{The main result} \label{section main} 

Recall that $\chi$ is the cusp map and that $\phi$ is defined in \eqref{def}. The map $g$ appearing in the formula below plays an inert role, and is just designed 
to ensure that $\Phi$ is non-degenerate; we can take, for example $g (z_2) = z_2$. This seems to mean that non-degeneracy is not the only issue in the question 
of estimating $\beta_{2}^{+} (C_\Phi)$. 
\smallskip

Our example  appears as a perturbation of the diagonal map defined by $\Delta (z_1, z_2) = \big( \chi (z_1), \chi (z_1) \big)$ for which we already know 
(\cite[Theorem~2.4]{DHL-surjective}) that $\Delta (1, 1) = (1, 1)$ and  $\beta_{2}^{+} (C_\Delta) < 1$. This map is degenerate, but the perturbation clearly 
gives a non degenerate one since its Jacobian is $J_{\Phi} (z_1, z_2) = c \, (\varphi \circ \chi) (z_1) \, \chi ' (z_1) \, g' (z_2)$.  

\begin{theorem} \label{main}  
Let:
\begin{displaymath} 
\Phi (z_1, z_2) = \big( \chi (z_1), \chi (z_1) + c \, (\varphi \circ \chi) (z_1) \, g (z_2) \big)
\end{displaymath} 
be the function defined in \eqref{def-Phi}. 

Then: 
\begin{enumerate}
\setlength\itemsep{2 pt}

\item [$1)$] $\Phi (\D^2)\subseteq \D^2$ and $C_\Phi \colon H^{2} (\D^2) \to H^{2} ( \D^2)$ is compact;
 
\item [$2)$] $\Phi$ is non degenerate, and its components belong to the bidisk algebra; 

\item [$3)$] $\Phi (\T^2) \cap {\T^2} = \Phi (\T^2) \cap \partial{}_{e}\D^2 \neq\emptyset$; 

\item [$4)$] $a_{n^2} (C_\Phi) \lesssim \exp(- \tau n)$, for some $\tau > 0$, implying $\beta_{2}^{+} (C_\Phi) < 1$.
 \end{enumerate} 
\end{theorem}
\begin{proof}
That $\Phi$ maps $\D^2$ to itself is proved  in Lemma~\ref{eins} and that the composition operator $C_\Phi \colon H^{2} (\D^2) \to H^{2} ( \D^2)$ 
is compact, in Corollary~\ref{Phi-HS}. 
Item $2)$ is due to the presence of $g$, as explained above. The fact that $\Phi (\T^2) \cap {\T^2} \neq\emptyset$ is clear since $\Phi (1, 1) = (1, 1)$. 
It remains to prove $4)$.

Once more, the proof will be conveniently divided into several steps. We begin by a lemma which is in fact obvious, but explains well what is going on.

\begin{lemma} \label{0} 
Let $\lambda = 1 - \frac{\sigma^{j_0}}{2K}$, where $\sigma$, $K$ and $j_0$ are as in \eqref{sigma}, Proposition~\ref{cusp},\,$2)$, and \eqref{j-zero}. 
Let $r_n = 1 - \frac{1}{n}$, and let  $\mu_1$, $\mu_2$,  $\mu_3$  the respective restrictions of $m_\Phi$ to  the disk $\overline{\lambda \D^2}$, the annulus 
$r_n\D^2 \setminus \overline{\lambda \D^2}$,  and the annulus $\D^2 \setminus r_n \D^2$. We then have:
\begin{displaymath} 
C_\Phi = T_1 \oplus T_2 \oplus T_3 \, ,
\end{displaymath} 
where $T_j$ is the canonical injection of $H^{2} (\D^2)$ into $L^{2} (\mu_j)$.
\end{lemma}

This is indeed obvious since:
\begin{displaymath} 
\Vert C_{\Phi} f \Vert^2 = \int_{\D^2} |f|^2 dm_\Phi \, ,
\end{displaymath} 
and by splitting the integral into three parts.
\smallskip

We now majorize separately the numbers $a_{p} (T_j)$,  for $j = 1, 2, 3$. In the sequel, the positive constant $\tau$ may vary from one formula to another. 
\smallskip

{\sl Step 1.} It holds:
\begin{equation} \label{T-un}
a_{n^{2}} (T_1) \lesssim \e^{- \tau n} \, .
\end{equation} 
\begin{proof} 
Let $V = z_{1}^n H^{2} (\D^2) + z_{2}^n H^{2} (\D^2)$; this is a subspace of $H^{2} (\D^2)$ of codimension $\leq n^2$, since:
\begin{displaymath} 
V = \big\{ f \in H^{2} (\D^2) \, ; \ D_{1}^{\, j} D_{2}^{\, k} f (0, 0) = 0 \quad \text{for } 0 \leq j, k < n \big\} \, .
\end{displaymath} 
If $f (z) = \sum_{\max(j, k) \geq n} a_{j, k} \, z_{1}^{j} z_{2}^{k}\in V$ and $\Vert f \Vert_2 = 1$, one can write:
\begin{displaymath} 
f (z) = z_{1}^n  \, q_{1} (z_1, z_2) + z_{2}^n \, q_{2} (z_1, z_2) \, ,
\end{displaymath} 
with: 
\begin{displaymath} 
q_{1} (z) = \sum_{j \geq n, k \geq 0} a_{j, k} \, z_{1}^{j - n} z_{2}^{k} 
\quad \text{and} \quad q_{2} (z) = \sum_{j < n, k \geq n} a_{j, k} \, z_{1}^{j} z_{2}^{k - n} \, ,
\end{displaymath} 
which satisfy  $\Vert q_j \Vert_2 \leq \Vert f \Vert_2 = 1$, $j = 1,2$. 

An easy estimate now gives (since $ \max (|z_{1}|^n,  |z_{2}^n|) \leq \lambda^{n}$ on $\overline{\lambda \D^2}$):
\begin{align*}
\Vert T_{1}f \Vert^2 
& \leq 2 \, \bigg( \int_{\lambda \D^2} \big( |z_{1}^{n}|^2 |q_{1} (z_1, z_2)|^2  + |z_{2}^{n}|^2 |q_{2} (z_1, z_2)|^2\big) \, dm_\Phi \bigg) \\
& \lesssim \lambda^{2 n}\int_{\lambda \D^2} \big(|q_1|^2 +|q_2|^2 \big) \, dm_\Phi 
\lesssim \lambda^{2n} \big(\Vert q_1 \Vert_{2}^{2} + \Vert q_2 \Vert_{2}^{2}\big) \lesssim \lambda^{2 n}  \, ,
\end{align*}
since we  know by Corollary~\ref{Phi-HS} that $C_\Phi$ is bounded on $H^{2}(\D^2)$ and hence that $m_\Phi$ is a Carleson measure for $H^{2}(\D^2)$.  
Alternatively, we could majorize $|q_{j}(z_1, z_2)|$ uniformly on the support of $\mu_1$. We hence obtain:
\begin{equation} \label{unus} 
a_{n^2 + 1}(T_1) = c_{n^2 + 1} (T_1) \lesssim \e^{- \tau n} \, . \qedhere
\end{equation} 
\end{proof}
\smallskip

{\sl Step 2.} It holds:
\begin{equation} \label{T-trois}
a_{n^{2}} (T_3) \lesssim \e^{- \tau n} \, .
\end{equation} 
\begin{proof}
In one variable, we could use the Carleson embedding theorem; but this theorem for the bidisk and the Hardy space $H^{2}(\D^2)$ notably has a more 
complicated statement (\cite{CH}; see also \cite{Feff}), and cannot be used efficiently here. Our strategy  will be to replace it by a sharp estimation of a 
Hilbert-Schmidt norm.  

We set $h_n = 1 - r_n = 1/ n$.

Clearly, denoting by $S_2$ the Hilbert-Schmidt class:
\begin{displaymath} 
\Vert T_3 \Vert^2 \leq \Vert T_3 \Vert_{S_2}^{2} = \int \frac{d \mu_{3} (w)}{(1 - |w_1|^2)(1 - |w_2|^2)} 
\leq \int \frac{d\mu_{3} (w)}{(1 - |w_1|)(1 - |w_2|)} \, \cdot
\end{displaymath} 
Now, if $w = (w_1, w_2) = \big( \chi (z_1), \chi (z_1) + c \, (\varphi \circ \chi) (z_1) \, g (z_2))$ belongs to the  support of $\mu_3$, we have 
$\max (|w_1|,|w_2|) \geq r_n = 1 - h_n$, and, recalling \eqref{deux}: 
\begin{equation} \label{recall} 
|w_1| \geq 2 \, |w_2| - 1 \, ,
\end{equation} 
we have in either case $| w_1 | \geq 1 - 2 h_n$. By Proposition~\ref{cusp},\,$2)$, this implies that:
\begin{displaymath} 
| 1 - w_1 | \leq 2 \, K  h_n\, .
\end{displaymath} 
Corollary~\ref{Phi-HS} gives: 
\begin{align*}
\Vert T_3 \Vert^2 
& \lesssim \int_{|\chi (\e^{i t_1}) - 1| \leq 2 K h_n} 
\frac{dt_1 \, dt_2}{(1 - |\chi (\e^{i t_1})|)(1 - |\chi (\e^{i t_1}) + c\, (\varphi \circ \chi) (\e^{i t_1}) \, g (\e^{i t_2})|)}  \\ 
& = I (2 K h_n) \lesssim \e^{- \tau / h_n} \, . \phantom{a^{A^{B^C}}}
\end{align*}
But $h_n = 1/n$, so that:
\begin{equation} \label{duo} 
a_{n^{2}} (T_3) \leq \Vert T_3 \Vert \lesssim \e^{- \tau n} \, . \qedhere
\end{equation}
\end{proof}
\smallskip

{\sl Step 3.} It holds:
\begin{equation} \label{T-deux} 
a_{n^2} (T_2) \lesssim \e^{- \tau n} \, .
\end{equation} 

This estimate follows from the following key auxiliary lemma. In fact, this lemma will give, for the Gelfand numbers, $c_{n^{2}} (T_2) \lesssim \e^{- \tau n}$, 
and we know that they are equal to the approximation numbers.

Let $M \colon H^{2} (\D^2) \to {\rm Hol} (\D)$ be the linear map defined by:
\begin{displaymath} 
M f (z) = f (z, z) \, , 
\end{displaymath} 

Recall that $a_j = 1 - \sigma^j$ and $N_n = \Big[\frac{\log 2n}{\theta\log 1/\sigma} \Big] + 1$. 
\begin{lemma} \label{cle} 
Let $E$ be the closed subspace of $H^{2} (\D^2)$ defined by: 
\begin{displaymath}
\begin{split} 
E = \Big\{ f \in H^{2} (\D^2) \, ; \  \big[ M (D_{2}^{\, k} & f) \big]^{(l)} (a_j) = 0 \\ 
& \text{ for }\, 0 \leq l < n,\ 0 \leq k \leq m_j,\ 1 \leq j \leq N_n\Big\} \, .
\end{split}
\end{displaymath} 

Then, we can adjust the numbers $m_j$ so as to guarantee that, for some positive constant $\tau$:
\begin{displaymath} 
{\rm codim}\, E \lesssim n^2
\end{displaymath} 
and, for all $f \in E$ with $\Vert f \Vert_2 \leq 1$:
\begin{displaymath} 
\Vert T_{2} (f) \Vert_{2} \lesssim \exp (- \tau n)  \, .
\end{displaymath} 
\end{lemma}
 \goodbreak
\begin{proof}
This is the most delicate part. 

Recall that:
\begin{displaymath} 
h_n = 1 / n \, ,\quad   r_n = 1 - h_n \, , \quad  \lambda = 1 - \frac{\sigma^{j_0}}{2 \, K} \, \cdot  
\end{displaymath} 
We need a uniform estimate of $|f (w)|$ for $f \in E$ with $\Vert f \Vert_2 \leq 1$ and for:
\begin{displaymath} 
w = (w_1,w_2) \in {\rm supp}\, m_\Phi \cap (r_n \D^2 \setminus \overline{\lambda \D^2}) \, .  
\end{displaymath} 
This estimate will be given by  Lemma ~\ref{drei}, Lemma~\ref{vier} and Lemma~\ref{funf}. Note that we have:
\begin{displaymath} 
\chi (z_1) \in \D \setminus [ S (1, 1/n) \cup \overline{D} (0, 2 \lambda - 1) ] \, .
\end{displaymath} 
Indeed, if $(w_1, w_2) = \Phi (z_1, z_2) \notin \lambda \D^2$, we have $\max (|w_1|, |w_2|) > \lambda$; so either 
$|w_1| > \lambda \geq 2 \lambda - 1$, or $|w_2| > \lambda$ and again $|w_1| > 2 \lambda - 1$ since $|w_1| \geq 2 \, |w_2| - 1$, by 
\eqref{deux}. Hence $w_1 \notin \overline{D} (0, 2 \lambda - 1)$.  Moreover, we have  $|1 - w_1|\geq 1 - |w_1| >  1 / n$, so 
$w_1 \notin S (1, 1 / n)$. 

Using Lemma~\ref{drei}, select $j_{0} \leq j \leq N_n$ such that $|\chi (z_1) - a_j| \leq \frac{1}{4} (1 - a_j)$. Now set:
\begin{displaymath} 
A = \big( \chi(z_1), \chi (z_1)) \quad  \text{and} \quad  \Delta = \big( 0,  (\varphi \circ \chi) (z_1) \, g (z_2) \big) \, . 
\end{displaymath} 

Our strategy will be the following. We write:
\begin{align*}
f  [\Phi (z_1,z_2) ] = f (A +\Delta) 
& =\sum_{k = 0}^\infty \frac{D_{2}^{\, k} f (A)}{k!} \, \Delta^k \\
& =\sum_{k = 0}^\infty \frac{M (D_{2}^{\, k}f) [ \chi (z_1) ]}{k!} \, \Delta^k 
=\sum_{k = 0}^\infty \frac{h_k [ \chi (z_1) ]}{k!} \, \Delta^k \, ,
\end{align*}
with $h_k = M (D_{2}^{\, k} f)$, and we put:
\begin{displaymath} 
S_j = \sum_{k = 0}^{m_j} \frac{h_k [\chi (z_1)]}{k!} \, \Delta^k
\end{displaymath} 
and 
\begin{displaymath} 
\phantom{M}  R_j = \sum_{k > m_j} \frac{h_k [\chi (z_1) ] }{k!} \, \Delta^k \, .
\end{displaymath} 

We will estimate separately $S_j$ and $R_j$.
\medskip\goodbreak

a) \emph{Estimation of $R_j$.}

Recall that $j$ is such that $j_{0} \leq j \leq N_n$ and $|\chi (z_1) - a_j| \leq \frac{1}{4}\, (1 - a_j)$.
We saw in the proof of this Lemma~\ref{drei} that $1 - |\chi (z_1)| \leq |1 - \chi (z_1) \leq \frac{5}{4} \, \sigma^j$. Hence:
\begin{displaymath} 
|\Delta| \leq |(\varphi \circ \chi) (z_1)| \leq \exp \bigg( - \frac{\delta}{|1 - \chi (z_1)|^\theta} \bigg) \lesssim \exp (- \tau \, \sigma ^{- j\theta} )  \, .
\end{displaymath} 

Now, use Lemma~\ref{vier} and \eqref{ob} to get: 
\begin{align*}
|R_j| 
& \leq \sum_{k > m_j} \frac{2^{k + 1}}{(1 - |\chi (z_1)|)^{k + 1}} \, |\Delta|^k 
\lesssim \sum_{k > m_j} 2^{k} \sigma^{- j k} \exp (- \tau k \, \sigma^{- j \theta}) \\
& \lesssim 2^{m_j} \sigma ^{- j m_j} \exp (- \tau m_j \sigma^{- j \theta}) \lesssim \exp (- \tau m_j \sigma^{- j \theta}) 
\end{align*}
for some absolute constant $\tau > 0$, that is:
\begin{equation} \label{R} 
|R_j| \lesssim \exp (- \tau n) 
\end{equation}
if we take:
\begin{equation} \label{adj}  
m_j = [n \, \sigma^{j \theta}] + 1 \, .
\end{equation} 
\smallskip

b) \emph{Estimation of $S_j$.} 

We saw in the estimation of $R_j$ that $1 - |\chi (z_1)| \gtrsim \sigma^j$. Now, remember that $h_{k}^{(l)} (a_j) = 0$ for $l < n$, since $f \in E$, we then  use
Lemma ~\ref{funf} to get, when we take the values: 
\begin{displaymath} 
a = a_j \, , \quad  1 - a_j = \sigma^j \, , \quad  b = \chi (z_1) \, , \quad \rho = \frac{1}{2} \, \raise 1,5 pt \hbox{,}
\end{displaymath} 
a good upper bound for $\frac{h_k [ \chi (z_1) ]}{k!}$ when $k \leq m_j$, namely:
\begin{displaymath} 
\bigg| \frac{h_{k} [\chi (z_1) ]}{k!} \bigg| \lesssim \frac{4^{k + 1}}{\sigma^{j (k + 1)}} \, \rho^n \, . 
\end{displaymath} 
We then obtain an estimate of the form:
\begin{align*} 
|S_j| 
\lesssim \sum_{k = 0}^{m_j} \rho^n  \frac{4^{k + 1}}{\sigma^{j (k + 1)}} 
& \lesssim \rho^n \frac{4^{m_j}}{\sigma^{j m_j}}
= \exp \Big( - n \log 2 + m_j \log 4 - j m_j \log \frac{7}{8} \Big) \\ 
& \lesssim  \exp (- 4 \tau n + B j m_j) 
\end{align*} 
with $\tau = \frac{1}{4}\, \log 2$ and $B \leq \log 4 + \log (8 / 7) \leq 2$;  or else, using (\ref{adj}): 
\begin{displaymath} 
|S_j| \lesssim \exp (- 4 \tau  n + B j n \, \sigma^{j \theta} + B j) \, .
\end{displaymath} 
But since $\sigma = 7 / 8 < 1$, the implied exponent, for $j_{0}\leq j \leq N_n$:
\begin{displaymath} 
- 4\tau  n +  B j n \, \sigma^{j \theta} + B j = n (- 4 \tau + B j \, \sigma^{j \theta}) + B j \, ,
\end{displaymath} 
is $\leq  - 2\tau n + B' \log n$, provided that we choose $j_0$ large enough, namely such that $j_0 \big( \frac{7}{8} \big)^{j_0 \theta} \leq 1 / 4$. This 
implies an inequality of the form:
\begin{equation} \label{S} 
|S_j| \lesssim \e^{- 2 \tau  n} n^{B'}\lesssim \e^{- \tau n} \, .
\end{equation} 

Putting the estimates \eqref{R} and \eqref{S} on $R_j$ and $S_j$ together, we obtain, for every $f\in E$ with $\Vert f \Vert_2 \leq 1$:
\begin{equation} \label{b} 
\Vert T_{2} f \Vert \lesssim \e^{- \tau n} \, . 
\end{equation}

It remains to bound from above the codimension of $E$. Since $N_n = \big[ \frac{\log 2n}{\theta \log 1/\sigma}\big] + 1 $ with $\sigma = 7 /8 $ 
and $m_j = [n \, \sigma^{j \theta}] + 1$, we see that:
\begin{displaymath} 
{\rm codim}\, E \leq \sum_{l = 0}^{n - 1}\sum_{j = 1}^{ N_n} m_j \leq \sum_{l = 0}^{n - 1} \sum_{j = 1}^{ N_n} (n \, \sigma^{j \theta} + 1)
\lesssim n^2 \sum_{j = 1}^\infty \sigma^{j \theta} + n \log n <  q \, n^2 \, .
\end{displaymath} 
Therefore \eqref{b} can be read as well, remembering the equality of approximation numbers and  Gelfand numbers: 
\begin{equation} \label{tres}   
a_{q \, n^2} (T_2) = c_{q \, n^{2}} (T_2)  \lesssim \e^{- \tau n}.  
\end{equation}

Putting the estimates \eqref{unus}, \eqref{duo}, and \eqref{tres} together ends the proof of Lemma~\ref{cle}.
\end{proof}
\smallskip

Finally, Lemma~\ref{0} and \eqref{sa} give:
\begin{displaymath}
a_{3 n^2} (C_\Phi) = a_{3 n^2} (T_1 \oplus T_2 \oplus T_3) \leq a_{n^2} (T_1) + a_{n^2} (T_2) + a_{n^2 } (T_3) 
\lesssim  \e^{- \tau n} \, ,
\end{displaymath}
thereby finishing the proof of Theorem~\ref{main}.
\end{proof}
%

%%%%%%%%%%%%%%%%%%%%%%%%%%%%%%%%%%%%%%%%%%%%%%%%%%%%%%%%%%%%%%%%%%%%%%%%%%%%%%%%%
\bigskip

\noindent{\bf Acknowledgments.} This work was initiated during a stay of the first and second-named authors in the University of Sevilla in February 2018. 
They wish to thank their Spanish colleagues for their kindness, and for the excellent working conditions which they provided.
\medskip

L. Rodr{\'\i}guez-Piazza is partially supported by the project MTM2015-63699-P (Spanish MINECO and FEDER funds).
\goodbreak

%%%%%%%%%%%%%%%%%%%%%%%%%%%%%%%%%%%%%%%%%%%%%%%%%%%%%%%%%%%%%%%%%%%%%%%%

 %%%%%%%%%%%%%%%%%%%%%%%%%%%%%%%%%%%%%%%%%%%%%%%%%%%%%%%%%%
\smallskip\goodbreak

{\footnotesize
Daniel Li \\ 
Univ. Artois, Laboratoire de Math\'ematiques de Lens (LML) EA~2462, \& F\'ed\'eration CNRS Nord-Pas-de-Calais FR~2956, 
Facult\'e Jean Perrin, Rue Jean Souvraz, S.P.\kern 1mm 18 
F-62\kern 1mm 300 LENS, FRANCE \\
daniel.li@univ-artois.fr
\smallskip

Herv\'e Queff\'elec \\
Univ. Lille Nord de France, USTL,  
Laboratoire Paul Painlev\'e U.M.R. CNRS 8524 \& F\'ed\'eration CNRS Nord-Pas-de-Calais FR~2956 
F-59\kern 1mm 655 VILLENEUVE D'ASCQ Cedex, FRANCE \\
Herve.Queffelec@univ-lille.fr
\smallskip
 
Luis Rodr{\'\i}guez-Piazza \\
Universidad de Sevilla, Facultad de Matem\'aticas, Departamento de An\'alisis Matem\'atico \& IMUS,  
Calle Tarfia s/n \\ 
41\kern 1mm 012 SEVILLA, SPAIN \\
piazza@us.es
}

\end{document}